\documentclass[12pt,oneside,british,a4wide]{amsart}
\usepackage{amsmath, amssymb,verbatim,appendix}
\usepackage[mathscr]{eucal}
\usepackage{amscd}
\usepackage{amsthm}
\usepackage{stmaryrd}
\usepackage{enumerate}
\usepackage{comment}
\usepackage[latin1]{inputenc}
\usepackage{tikz}
\usetikzlibrary{shapes,arrows}
\usepackage{cite}
\usepackage{url}

\usepackage{setspace,a4wide,color,xcolor,graphicx}

\def\showauthornotes{1}

\ifnum\showauthornotes=1
\newcommand{\Authornote}[2]{{\sf\small\color{red}{[#1: #2]}}}
\else
\newcommand{\Authornote}[2]{}
\fi

\newtheorem{theorem}{Theorem}
\newtheorem{lemma}{Lemma}
\newtheorem{corollary}{Corollary}
\newtheorem{proposition}{Proposition}

\theoremstyle{definition}

\newtheorem{definition}{Definition}

\def\E{\mathbb{E}}
\def\Z{\mathbb{Z}}

\def\C{\mathbb{C}}

\def\F{\mathbb{F}}

\def\e{\epsilon}

\newenvironment{proofof}[1]{\indent{\scshape Proof of #1}:~~}{\qed}

\begin{document}

\title[Quantitative structure of stable sets in finite abelian groups]{Quantitative structure of stable sets\\in finite abelian groups}

\author{C. Terry}

\author{J. Wolf}

\address{Department of Mathematics, University of Maryland, College Park,
MD 20742, USA}

\email{cterry@umd.edu}

\address{School of Mathematics, University of Bristol, Bristol BS8 1TW, UK}

\email{julia.wolf@bristol.ac.uk}

\date{}

\begin{abstract}
We prove an arithmetic regularity lemma for stable subsets of finite abelian groups, generalising our previous result for high-dimensional vector spaces over finite fields of prime order. A qualitative version of this generalisation was recently obtained by the first author in joint work with Conant and Pillay, using model-theoretic techniques. In contrast, the approach in the present paper is highly quantitative and relies on several key ingredients from arithmetic combinatorics.
\end{abstract}

\maketitle


\section{Introduction}\label{sec:intro}

Motivated by model-theoretic considerations and the graph-theoretic result \cite{Malliaris:2014go}, the authors proved an arithmetic regularity lemma for stable subsets of high-dimensional vector spaces over a fixed finite field of prime order in \cite{Terry:5CFoQi42}. In order to be able to state it, we recall the definition of a $k$-stable set from \cite{Terry:5CFoQi42}. Throughout this paper, $G$ will be a general finite abelian group unless explicitly stated otherwise.

\begin{definition}[$k$-stable subset]\label{def:stableset}
Let $A\subseteq G$, and let $k\geq 2$ be an integer. Then $A$ is said to have the \emph{$k$-order property} if there exist sequences $a_1,a_2,\dots,a_k$ and $b_1,b_2,\dots,b_k$ in $G$ such that $a_i+b_j\in A$ if and only if $i\leq j$. A set $A\subseteq G$ is said to be $k$-stable if it does not have the $k$-order property.\end{definition}

The following result appeared in \cite{Terry:5CFoQi42}.

\begin{theorem}[Stable arithmetic regularity in vector spaces]\label{thm:mainff}
For any $k\geq 2$, $0<\e<1$, and any prime $p$, there is $m = m(k,\epsilon,p)$ such that the following holds. Suppose that $G:=\mathbb{F}_p^n$ with $n\geq m$, and that $A\subseteq G$ is $k$-stable. Then there is a subspace $H\leqslant G$ of codimension at most $O_k(\epsilon^{-O_k(1)})$ such that for any $g\in G$, either $|(A-g)\cap H|\leq \epsilon |H|$ or $|H\setminus (A-g)|\leq \epsilon|H|$.
\end{theorem}

For a discussion of the significance of the arithmetic regularity in arithmetic combinatorics, and the relationship between Theorem \ref{thm:mainff} and Green's arithmetic regularity lemma \cite{Green:2005kh}, we refer the reader to the introduction of \cite{Terry:5CFoQi42}.

Theorem \ref{thm:mainff} was built upon by two groups of authors in fairly rapid succession. Alon, Fox and Zhao \cite{Alon:2018va} relaxed the requirement that $A$ be stable to the assumption that its set of translates $\{g+A: g\in G\}$ have bounded VC-dimension. This is a natural extension to seek from a model-theoretic point of view, but the result obtained is necessarily weaker than the stable version in Theorem  \ref{thm:mainff}: one cannot expect uniformity to hold for \emph{all} cosets of $H$ but must make do with a $(1-\e)$-proportion. The main result in \cite{Alon:2018va} is valid in abelian groups of bounded exponent, and we shall not state it explicitly here. 

A result along the same lines was obtained shortly afterwards by Sisask \cite{Sisask:2018wz}, using an entirely different set of techniques. His regularity lemma for sets of bounded VC-dimension (in the sense defined above) is valid in all finite abelian groups. We shall discuss the main ingredient of his proof in more detail in Section \ref{sec:almostper}. 

In a general finite abelian group $G$, which may not contain any non-trivial proper subgroups (such as the cyclic group $\Z/p\Z$ with $p$ a prime), we need to replace the subspace in the statement of Theorem \ref{thm:mainff} with a so-called \emph{Bohr set}. Recall that given a set $K\subseteq \widehat{G}$ of characters on $G$ and any $\rho>0$, a Bohr set $B$ of width $\rho$ and frequency set $K$ is defined as
\begin{equation*}
B(K,\rho):=\{x\in G: |\gamma(x)-1|\leq \rho \text{ for all $\gamma\in K$}\}.
\end{equation*}
The size $|K|$ of the frequency set is often referred to as the \emph{rank} of the Bohr set. It is easy to see that when $G:=\F_p^n$ and $0<\rho<1/p$, a Bohr set $B(K,\rho)$ is simply a subspace of $\F_p^n$ of codimension at most $|K|$. In the more general setting, Bohr sets act as approximate subgroups in the sense that they are somewhat closed under addition (see Section \ref{sec:regBohr}). 

Shortly after the preprint \cite{Terry:5CFoQi42} appeared, the first author, in collaboration with Conant and Pillay \cite{Conant:2017uw}, obtained the following qualitative arithmetic regularity lemma for stable subsets of general finite groups.

\begin{theorem}[Conant-Pillay-T., 2017+]\label{thm:CPT}
For any $k\geq 1$ and $0<\e<1$, there are $n=n(k,\e)$ and $N=N(k,\e)$ such that the following holds. Suppose that $G$ is a finite group of order at least $N$, and that $A\subseteq G$ is $k$-stable. Then there is a normal subgroup $H \leqslant G$ of index at most $n$ such that for each $g\in G$, either $|gH\cap A|\leq \e|H|$ or $|gH\setminus A|\leq \e|H|$.
\end{theorem}

The proof of Theorem \ref{thm:CPT} in \cite{Conant:2017uw} relies entirely on model-theoretic techniques, specifically the ``localisation" of existing results by Hrushovsky and Pillay \cite{Hrushovski:1994fn} in stable group theory. The authors also obtained a stronger structural result than that proved in \cite{Terry:5CFoQi42} (see the discussion in Section \ref{sec:struc}). In a separate paper \cite{Conant:2018tx}, Conant, Pillay and Terry also proved a result for subsets of finite groups whose left-translates have bounded VC-dimension. In this setting the model-theoretic techniques were less well developed, and established in the companion paper \cite{Conant:2018tt}.

The main result we obtain in this paper is a quantitative version of Theorem \ref{thm:CPT} in the abelian case. 

\begin{theorem}[Main result]\label{thm:main}
For all $k\geq 2$ and $0<\e<1$, there is $N_0 = N_0(k,\epsilon)$ such that the following holds. Suppose that $G$ is a finite abelian group of order at least $N_0$, and that $A\subseteq G$ is $k$-stable. Then there exists a Bohr set $B$ of width at least $\epsilon^{O_k(1)}$ and rank at most $\epsilon^{-O_k(1)}$ such that for any $g\in G$, either $|(A-g)\cap B|\leq \epsilon |B|$ or $|B\setminus (A-g)|\leq \epsilon|B|$. 
\end{theorem}

The paper is organised as follows. In Section \ref{sec:regBohr} we give a concise introduction to the analytic techniques underlying the proof. These are entirely standard in arithmetic combinatorics, allowing many arguments in the toy setting of $\F_p^n$ to be transferred to general finite abelian groups via a process known as \emph{Bourgainisation}. However, a key part of the argument in \cite{Terry:5CFoQi42} does not yield to generalisation in this way: it does not seem possible to show that a pair of Bohr sets gives rise to a sufficiently ``regular pair" in the relevant Cayley graph \cite[Lemma 7]{Terry:5CFoQi42}.

As a work-around for this situation we use a refinement of a recent almost-periodicity result for bounded VC-dimension by Sisask \cite{Sisask:2018wz}. In Section \ref{sec:almostper} we give some background on the versatile Croot-Sisask almost-periodicity technique \cite{Croot:2010fk}, and establish the main ingredient in our proof, namely the fact that stable sets are dense on some translate of a Bohr set with suitably bounded parameters. In Section \ref{sec:mainproof} we use this result to build a model-theoretic tree that witnesses instability in the absence of efficient regularity, by carrying out an approximate version of the argument in \cite{Terry:5CFoQi42} (and ultimately \cite{Malliaris:2014go}). This will conclude the proof of Theorem \ref{thm:main}.

Together with a simple lemma given at the start of Section \ref{sec:struc}, Theorem \ref{thm:main} allows us to deduce a structural corollary for two classes of finite abelian groups that are of particular interest in arithmetic combinatorics. First, we recover \cite[Corollary 1]{Terry:5CFoQi42}, which states, roughly speaking, that stable subsets of $\F_p^n$ look approximately like a union of cosets of a subspace of bounded codimension. Second, we find that there are essentially no non-trivial stable sets in cyclic groups of prime order, providing a model-theoretic justification for the common experience in arithmetic combinatorics that these cyclic groups are substantially more difficult to handle.

In order to establish a structural result for general finite abelian groups (Corollary \ref{cor:gengroup} in Section \ref{sec:boot}), we bootstrap the existing argument to conclude that when $A$ is stable, efficient regularity holds not only with respect to a Bohr set as in Theorem \ref{thm:main}, but with respect to a subgroup of bounded index. 

\begin{theorem}[Main result: subgroup version]\label{thm:mainintro2}
For all $k\geq 2$ and $\mu\in (0,1)$, there is $M_0=M_0(k,\mu)$ such that the following holds. Suppose that $G$ is a finite abelian group of order at least $M_0$, and that $A\subseteq G$ is $k$-stable. Then there is a subgroup $H\leq G$ of index at most $\exp(\mu^{-O_k(1)})$ such that for any $g\in G$, either $|(A-g)\cap H|\leq \mu |H|$ or $|H\setminus (A-g)|\leq \mu|H|$.
\end{theorem}

It is likely that a more direct proof of Theorem \ref{thm:mainintro2} avoiding this bootstrapping can be found. Having said this, the fact that we run the tree construction a second time does not affect the quantitative outcome of the argument, which is thus efficient in that sense.

\vspace{8pt}

\textbf{Acknowledgements.} The authors wish to thank the Simons Institute for the Theory of Computing at UC Berkeley for providing outstanding working conditions and a springboard for the authors' collaboration during the \emph{Pseudorandomness} programme in Spring 2017. The second author is grateful to Tim Gowers for sharing a helpful example with her.

\section{Analytic background}\label{sec:regBohr}

Throughout, let $G$ be a finite abelian group. As is common in arithmetic combinatorics, given a subset $B\subseteq G$ we define its \emph{characteristic measure} $\mu_B$ by $\mu_B(x):=1_B(x)|G|/|B|$ for all $x\in G$, where $1_B$ is the indicator function of $B$. Note that the normalisation is chosen so that $\E_{x \in G} \mu_B(x)=1$, where $\E_{x \in G}$ denotes the sum over all elements $x\in G$, normalised by the order of $G$. 

Let $\widehat{G}$ denote the group of characters of $G$, that is, the set of homomorphisms from $G$ into the multiplicative group on $\mathbb{T}=\{z\in \mathbb{C}: |z|=1\}$.  Given $f:G\rightarrow \mathbb{C}$ and $\gamma\in \widehat{G}$, the \emph{Fourier coefficient of $f$ at $\gamma$} is defined to be
\[\widehat{f}(\gamma)=\mathbb{E}_{x\in G} f(x)\gamma(x).\]
Given a set $X\subseteq G$, we denote by $\mathrm{Spec}_{\rho}(\mu_X)$ the set of characters $\gamma \in \widehat{G}$ at which $|\widehat{\mu_X}(\gamma)|\geq \rho$. It follows straight from the orthonormality of the characters that the \emph{convolution} $f*g:G\rightarrow\C$ of two functions $f,g:G\rightarrow\C$, defined by
\[f*g(x):=\E_{y\in G}f(y)g(x-y),\]
satisfies the relationship 
\[\widehat{f*g}(\gamma)=\widehat{f}(\gamma)\cdot\widehat{g}(\gamma)\]
for all $\gamma\in\widehat{G}$. These facts, together with the inversion formula $f(x)=\sum_{\gamma\in \widehat{G}}\widehat{f}(\gamma)\gamma(x)$ and Parseval's identity\footnote{Here the $\ell^2$ norms are defined with respect to normalised counting and counting measure on $G$ and $\widehat{G}$, respectively.} $\|f\|_2=\|\widehat{f}\|_2$, will come in useful in the proof of Proposition \ref{prop:propositionB}.
 
In contrast to \cite{Terry:5CFoQi42}, we only require a very limited amount of Fourier analysis in the present paper. We therefore content ourselves with simply stating that the Fourier transform of the function $f(x):=1_A(x)-(|A|/|G|)$ measures the extent to which $A$ is uniformly distributed on $G$. Moreover, by suitably adapting the function $f$ (see \cite[Definition 4]{Terry:5CFoQi42}), the Fourier transform can more generally provide a measure of how uniformly distributed $A$ is on translates of a subgroup of $G$.

The goal of this paper is to show that when $A$ is $k$-stable for some $k\geq 2$, then we can find a subgroup $B$ whose parameters depend only on $k$ and $\epsilon$ such that $A$ is well-distributed on every coset of $B$. In fact, we shall establish the existence of such a subgroup $B$ with an even stronger property, which we characterise as follows.

\begin{definition}[$\epsilon$-good]\label{def:good}
Let $A, B\subseteq G$, and let $y\in G$. We say that $y$ is \emph{$\epsilon$-good for $A$ with respect to $B$} if $|(A-y)\cap B|\leq \epsilon |B|$ or $|B\setminus (A-y)|\leq \epsilon |B|$. We say that $B$ is \emph{$\epsilon$-good for $A$} if $y$ is $\epsilon$-good for $A$ with respect to $B$ for all $y\in G$.
\end{definition}

Thus, when $B$ is $\e$-good for $A$ for some suitably small parameter $\e>0$, then $A$ is well distributed on every translate of $B$ in the sense that it either misses or fills this translate almost entirely.

As mentioned in the introduction, the cyclic groups $\Z/p\Z$ with $p$ a prime show that not every finite abelian group disposes of a non-trivial proper subgroup. In such cases we shall have to make do with a so-called Bohr set, which can replace a subgroup in many analytic arguments, using a set of techniques that is now standard in arithmetic combinatorics. Indeed, Green \cite{Green:2005kh} first proved his arithmetic regularity lemma in $\F_2^n$, and then followed the blueprint of this argument to prove the result in the more general context of finite abelian groups. We shall now give a summary of the main tools that enable this generalisation.\footnote{The field has evolved significantly since \cite{Green:2005kh}, and we shall adopt a more modern approach here.}  

Given $K\subseteq \widehat{G}$ and $\rho\in [0,1]$, define a \emph{Bohr set} $B(K,\rho)$ by
\begin{equation}\label{eq:defbohr}
B(K,\rho):=\{x\in G: |\gamma(x)-1|\leq \rho \text{ for all $\gamma\in K$}\}.
\end{equation}
We call $\rho$ the \emph{width} of $B=B(K,\rho)$, and refer to $|K|$ as the \emph{rank} of $B$. Observe that a Bohr set is always symmetric, i.e. if $B$ is a Bohr set then $x\in B$ if and only if $-x\in B$. It is easy to see that a Bohr set generalises the notion of a subspace of a finite-dimensional vector space over a finite field.  Indeed, if $G=\mathbb{F}_p^n$ for some prime $p$ and if $\rho<1/p$, then $B(K,\rho)$ is a subspace of $G$ whose orthogonal complement is spanned by $K$, and which is thus of codimension at most $|K|$. 

The following standard result, for a proof of which we refer the reader to \cite[Lemma 4.20]{Tao:2010rq}, states that Bohr sets of constant width and rank have positive density in $G$.
\begin{lemma}[Size of Bohr sets]\label{lem:sizefact}
For any $K\subseteq \widehat{G}$ and $\rho\in (0,1)$, $|B(K,\rho)|\geq \rho^{|K|}|G|$.
\end{lemma}

In order to be able to use a Bohr set as a substitute for a genuine subgroup, it needs to have a special property, known as \emph{regularity}.\footnote{This notion of regularity is unrelated to that making an appearance in the term ``regularity lemma".} Most of the definitions and results that follow go back to Bourgain \cite{Bourgain:1999vn}, and a good reference is \cite[Chapter 4]{Tao:2010rq} or \cite[Section 2]{Gowers:2011uq}. 

A Bohr set $B=B(K,\rho)$ is said to be \emph{regular} if for all $\epsilon \in (0,1/100|K|)$, 
\[|B(K,\rho(1+\epsilon))|\leq |B|(1+100|K|\epsilon) \text{ and }|B(K,\rho(1-\epsilon))|\geq |B|(1-100|K|\epsilon) .\]
The plentiful existence of regular Bohr sets was first noted by Bourgain \cite{Bourgain:1999vn}, see for instance \cite[Lemma 4.25]{Tao:2010rq}.
\begin{lemma}[Regular Bohr sets exist]\label{lem:regexist}
Given $K\subseteq \widehat{G}$ and $\rho>0$, there is $\rho'\in [\rho, 2\rho]$ such that $B(K,\rho')$ is regular.
\end{lemma}

For notational convenience we also use the following (non-standard) definitions from \cite{Gowers:2011uq}.

\begin{definition}[Closure of a regular Bohr set]\label{def:bohrint}
Given $\epsilon>0$ and Bohr sets $B=B(K,\rho)$, $B'=B(K,\sigma)$, we write $B'\prec_{\epsilon}B$ if $B$ and $B'$ are both regular, and $\sigma \in [\epsilon \rho/400|K|, \epsilon \rho/200|K|]$.  In this case we write $B^+:=B(K,\rho +\sigma)$ for the \emph{closure} of $B$.
\end{definition}

Observe that if $B'\prec_{\epsilon}B$, then by the triangle inequality, $B'+B\subseteq B^+$.  Combining this with the fact that $B$ is regular yields that 
\begin{equation}\label{eq:regfact1}
|(B'+B)\Delta B^+|=|B^+\setminus (B+B')|\leq |B^+\setminus B|\leq \epsilon |B|.
\end{equation}

We shall not require any more sophisticated properties of pairs $B'\prec_{\epsilon}B$ in this paper, and even the above is only used in Section \ref{sec:struc}.

Finally, as is standard in arithmetic combinatorics, we shall on occasion use the notation $f\ll g$ to mean that there exists a positive constant $C$ such that $f(x)\leq C g(x)$ for all $x\in G$.

\section{Stable sets are dense on Bohr sets}\label{sec:almostper}

In 2010, Croot and Sisask \cite{Croot:2010fk} introduced an important combinatorial technique establishing the existence of almost-periods of certain convolutions. This technique has found wide-reaching applications. In particular, it was instrumental in Sanders's proof of the Bogolyubov-Ruzsa lemma \cite{Sanders:2010wb} and the best known bounds on Roth's theorem \cite{Sanders:2011uq,Bloom:2014wx}.

The main idea holds for finite subsets $A$ and $B$ of \emph{any} group $G$ (not necessarily finite or abelian): given a set $S$ such that $|S\cdot A|\leq C|A|$, one can find a large subset $T$ of $S$ such that the convolution $\mu_A*1_B$ is approximately shift-invariant under all shifts by elements of the product set $TT^{-1}$, with the extent of the shift-invariance being measured by an $L^p$ norm for $p\in [2,\infty)$. With the help of a little Fourier analysis, on finite abelian groups the set of almost-periods can then be given the structure of a Bohr set. 

In recent work \cite{Sisask:2018wz}, Sisask revisited the above result under the additional assumption that the set of translates of $A$ has bounded VC-dimension.\footnote{In fact, Sisask's results are valid for a slightly more general notion of VC-dimension than the one we use here.} Recall that in this context the \emph{VC-dimension} of $A$ is defined to be the VC-dimension of its set of translates $\{g+A:g\in G\}$, that is, the maximum size of a set shattered\footnote{A family $\mathcal{F}$ of sets is said to \emph{shatter} $X$ if every subset $Y\subseteq X$ can be written as $Y=X\cap F$ for some $F\in\mathcal{F}$.} by $\{g+A:g\in G\}$. This assumption leads to a strengthening of the almost-periodicity conclusion in the sense that shift-invariance now holds pointwise (i.e. with the $L^\infty$ norm). We shall need the following version of \cite[Theorem 1.6]{Sisask:2018wz} that works relative to a Bohr set. 

\begin{proposition}[Relative almost periodicity for sets of bounded VC-dimension]\label{prop:propositionB}
Let $\eta, \rho, \sigma \in (0,1]$ and $d\in \mathbb{N}$.  Suppose that $A\subseteq G$ is a set of VC-dimension at most $d$, $B=B(K,\rho)$ is a Bohr set and $S$ is a subset of $B$ of relative density $\sigma=|S|/|B|$ with the property that $|A+S|\leq C|A|$. Then there exists a regular Bohr set $B'=B(K',\rho')\subseteq B$ with $|K'|\leq |K|+t'$, where
\[t'\ll\log (\sigma^{-1})+\eta^{-2}d\log(\eta^{-1})^2\log(C),\]
and $\rho'\gg \rho\eta/(|K|^2t')$,
such that for all $x\in G$, $t\in B'$,
\[|\mu_A*1_{-A}(x+t)-\mu_A*1_{-A}(x)|\leq \eta.\]
\end{proposition}

Before giving the proof of Proposition \ref{prop:propositionB}, we deduce the main result of this section, which will be a key ingredient in the proof of Theorem \ref{thm:main}.

\begin{proposition}[Sets of bounded VC-dimension are dense on Bohr sets]\label{prop:theoremA}
Let $\eta, \beta, \rho, \sigma \in (0,1]$ and $d\in \mathbb{N}$.  Suppose that $A\subseteq G$ is a set of VC-dimension at most $d$, and that $B=B(K,\rho)$ is a Bohr set with density $\beta$ and the property that $|A\cap B|\geq \sigma |B|$. Then there exists $x\in G$ and a regular Bohr set $B'=B(K',\rho')\subseteq B$ with 
$|K'|\leq |K|+t'$ and $\rho'\gg\rho\eta/(|K|^2t')$, where
\[t'\ll \log (\sigma^{-1})+\eta^{-2}d\log^2(\eta^{-1})\log((\sigma\beta)^{-1}),\]
such that
\[|A\cap (B'+x)|\geq (1-\eta)|B'|.\]
\end{proposition}

\begin{proofof}{Proposition \ref{prop:theoremA} assuming Proposition \ref{prop:propositionB}}
Given $A\subseteq G$ of VC-dimension at most $d$, and $B=B(K,\rho)$ of density $\beta$ such that $|A\cap B|\geq \sigma |B|$, set $S:=A\cap B$. Note that $S$ has density at least $\sigma$ on $B$, and that $|A+S|\leq |G|\leq(\sigma\beta)^{-1}|A|$. So we can apply Proposition \ref{prop:propositionB} with $C=(\sigma\beta)^{-1}$ to conclude that there exists a regular Bohr set $B'=B(K',\rho')\subseteq B$ with $|K'|\leq |K|+t'$ and $\rho'\gg\rho\eta/(|K|^2t')$, where
\[t'\ll \log (\sigma^{-1})+\eta^{-2}d\log^2(\eta^{-1})\log(C),\]
such that for all $x\in G$, $t\in B'$,
\[|\mu_A*1_{-A}(x+t)-\mu_A*1_{-A}(x)|\leq \eta.\]
In particular, since $\mu_A*1_{-A}(0)=1$, we find that $\mu_A*1_{-A}(t)\geq 1-\eta$ for all $t\in B'$. But averaging over all such $t$ implies that 
\[\mu_A*\mu_{B'}*1_{-A}(0)=\E_{t\in B'} \mu_A*1_{-A}(t)\geq 1-\eta,\]
where the left-hand side can be rewritten as
\[\mu_A*\mu_{B'}*1_{-A}(0)=\E_{a\in A} \mu_{B'}*1_A(a).\]
It follows that there exists $a\in A$ such that $1_A*\mu_{B'}(a)\geq 1-\eta$, which concludes the proof.
\end{proofof}

It is easy to see that if a set $A$ is $k$-stable according to Definition \ref{def:stableset}, then its VC-dimension is at most $k-1$ (see for example \cite[Lemma 6.2]{Sisask:2018wz}). Consequently, Proposition \ref{prop:theoremA} applies to stable subsets of $G$ and serves to replace \cite[Proposition 2]{Terry:5CFoQi42}, which was proved in the toy setting of high-dimensional vector spaces over prime fields. In fact, \cite[Proposition 2]{Terry:5CFoQi42} and its key ingredient \cite[Lemma 7]{Terry:5CFoQi42} are valid under the weaker assumption of bounded VC-dimension, but the proof of \cite[Lemma 7]{Terry:5CFoQi42} does not appear to generalise to finite abelian groups.

It is the following immediate corollary of Proposition \ref{prop:theoremA} that we shall use in the next section. The bounds we give are not best possible, but the exact shape does not ultimately make any difference to the main result we shall prove. What is important, however, is that they remain valid in the case where $K=\emptyset$ and/or $\rho=1$.  Unless otherwise stated, all logarithms are natural (to base $e$).

\begin{proposition}[Stable sets are dense on Bohr sets]\label{prop:theoremC}
There are constants $C_1\geq 1$ and $C_2>0$ such that the following holds.  Let $\epsilon\in (0,1/6)$, $\rho \in (0,1]$ and $k\in \mathbb{N}$.  Suppose that $A\subseteq G$ is $k$-stable and that $B=B(K,\rho)$ is a Bohr set with the property that $|A\cap B|\geq \epsilon |B|$. Then there exists $x\in G$ and a regular Bohr set $B'=B(K',\rho')\subseteq B$ with 
\begin{align*}
|K'|\leq C_1 (|K|k\e^{-3}\log(\rho^{-1}) +\log(\e^{-1}))\text{ and }\rho'\geq C_2\rho \e (C_1(|K|k\e^{-3}\log(\rho^{-1}) +\log(\e^{-1})))^{-3},
\end{align*}
such that 
\[|A\cap (B'+x)|\geq (1-\epsilon)|B'|.\]
\end{proposition}
\begin{proof}
By Proposition \ref{prop:theoremA} with $d=k-1$, $\eta=\e$ and $\sigma=|A\cap B|/|B|\geq \epsilon$, there exists an $x\in G$ and a regular Bohr set $B'=B(K',\rho')\subseteq B$ satisfying $|A\cap (B'+x)|\geq (1-\epsilon) |B'|$, $|K'|\leq |K|+t'$, and $\rho' \gg  \rho \e/|K|^2t'$, where
\[t'\ll\log (\epsilon^{-1}) + \epsilon^{-2}k\log^2(\epsilon^{-1})\log(\sigma^{-1}\beta^{-1}),\]
where $\beta=|B|/|G|$ is the density of $B$. By Lemma \ref{lem:sizefact}, we have $\beta \geq \rho^{|K|}$, and hence, on setting $m:=\lceil 1/\epsilon \rceil$,
\begin{align*}
t'\ll \log m +k(m\log m)^2 \log (m\rho^{-|K|})\ll k(m\log m)^2|K|\log (m/\rho)+\log m.
\end{align*}
It follows that
\[|K'|\leq |K|+t'\ll |K|(1+k(m\log m)^2\log(m/\rho))+\log m\ll |K|km^3\log (\rho^{-1})+\log m,\]
as desired.  Note also that $(|K|+t')^3\geq |K|^2t'$, so we have a lower bound on the width $\rho'$ of 
\[\rho'\gg \epsilon \rho (|K|+t')^{-3}\gg \epsilon \rho (|K|k m^3\log (\rho^{-1})+\log m)^{-3}.\] 
Substituting $\e^{-1}$ for $m$ and choosing $C_1,C_2$ appropriately yields the desired inequalities. 
\end{proof}

We now turn to the proof of Proposition \ref{prop:propositionB}.

\begin{proofof}{Proposition \ref{prop:propositionB}}
Suppose that $A\subseteq G$ is a set of VC-dimension at most $d$, $B=B(K,\rho)$ is a Bohr set and $S$ is a subset of $B$ of density $\sigma$ on $B$ with the property that $|A+S|\leq C|A|$. Now \cite[Theorem 1.5]{Sisask:2018wz} with $A=B$, $\e=\eta/4$ and $k=\lceil \log_2(2/\eta)\rceil$ yields a set $T\subseteq S$ of size at least $0.99\exp(-C'dk^2\log (C)/\eta^2)|S|$ for some constant $C'$ such that for all $x\in G$ and all $t\in kT-kT$,
\begin{equation}\label{eq:almostper}
|\mu_A*1_{-A}(x+t)-\mu_A*1_{-A}(x)|\leq \eta/4.
\end{equation}

For the remainder of the argument we follow the blueprint of the proof of \cite[Theorem 5.4]{Schoen:2014ua}, which establishes a relative $L^\infty$ almost-periodicity result for multiple convolutions. Fix an arbitrary $z\in T$, and set $X:=T-z$. It follows that (\ref{eq:almostper}) holds for all $x\in G$ and all $t\in kX\subseteq kT-kT$. Now by the triangle inequality and since $\E_y \mu_X^{(m)}(y)=1$ for any positive integer $m$,
\begin{equation}\label{eq:almostper2}
|\mu_A*1_{-A}*\mu_X^{(k)}(x)-\mu_A*1_{-A}(x)|\leq \eta/4
\end{equation}
for all $x\in G$, where for a function $f:G\rightarrow \C$, $f^{(m)}$ denotes the $m$-fold convolution of $f$ with itself.

Note that since $S$ is a dense subset of $B$, so is $T$. Specifically, $X$ is a subset of $B-z$ of relative density at least
\[0.99\exp(-C'dk^2\log(C)/\eta^2)\sigma.\]
Note also that translating $B$ by $-z$ does not affect the conclusion of \cite[Proposition 5.3]{Schoen:2014ua}, which is due to Chang and Sanders \cite[Proposition 4.2]{Sanders:2008dw}. Applying it to $T=X+z$, with $\nu=\eta/2$ and $\delta=1/2$, gives a set of characters $K''$ of size
\[|K''|\ll \log(2/\mu_B(T))\ll \log(1/ \sigma)+\eta^{-2}dk^2\log (C)\]
and a radius 
\[\rho'\gg \rho \eta/(|K|^2 \log(2/\mu_B(T))) \] 
such that 
\[|1-\gamma(t)|\leq \eta/2, \mbox{ for all } \gamma\in \mathrm{Spec}_{1/2}(\mu_X) \mbox{ and } t\in B'=B(K',\rho'),\]
where $K':=K\cup K''$. It can be checked that the parameters of $B'=B(K',\rho')$ satisfy the conclusions of Proposition \ref{prop:propositionB}. If $B'$ is not regular, it is easy to make it so at the cost of a factor of 2 in the width, by Lemma \ref{lem:regexist}.

Now for $x\in G$ and $t\in B'$, by the triangle-inequality
\[|\mu_A*1_{-A}*\mu_X^{(k)}(x+t)-\mu_A*1_{-A}*\mu_X^{(k)}(x)|\leq \sum_\gamma |\widehat{\mu_A}(\gamma)||\widehat{1_{-A}}(\gamma)||\widehat{\mu_X}(\gamma)|^k|\gamma(t)-1|.\]
The latter sum splits into two parts: when $\gamma \in\mathrm{Spec}_{1/2}(\mu_X)$, we have $|1-\gamma(t)|\leq \eta/2$, and when $\gamma \notin\mathrm{Spec}_{1/2}(\mu_X)$, then $|\mu_X(\gamma)|^k\leq 1/2^k\leq \eta/2$, from which it follows that 
\[\sum_\gamma |\widehat{\mu_A}(\gamma)||\widehat{1_{-A}}(\gamma)||\widehat{\mu_X}(\gamma)|^k|\gamma(t)-1|\leq \frac{\eta}{2} \sum_\gamma  |\widehat{\mu_A}(\gamma)||\widehat{1_{-A}}(\gamma)|.\]
This expression in turn is bounded above, using the Cauchy-Schwarz inequality, by
\[\frac{\eta}{2} \left(\sum_\gamma  |\widehat{\mu_A}(\gamma)|^2 \right)^{1/2}\left(\sum_\gamma |\widehat{1_{-A}}(\gamma)|^2 \right)^{1/2}=\frac{\eta}{2},\]
where the final equality follows from Parseval's identity.
Finally, we are able to bound 
\begin{align*}
|\mu_A*1_{-A}&(x+t)-\mu_A*1_{-A}(x)|\leq |\mu_A*1_{-A}(x+t)-\mu_A*1_{-A}*\mu_X^{(k)}(x+t)|\\+&|\mu_A*1_{-A}*\mu_X^{(k)}(x+t)-\mu_A*1_{-A}*\mu_X^{(k)}(x)|+|\mu_A*1_{-A}*\mu_X^{(k)}(x)-\mu_A*1_{-A}(x)|.
\end{align*}
The first and third terms in the sum are bounded above by $\eta/4$ using (\ref{eq:almostper2}), and we have just shown that the second term is at most $\eta/2$ for all $t\in B'$.
\end{proofof}

\section{Building a tree in the absence of efficient regularity}\label{sec:mainproof}

Our aim in this section is to prove the following. Throughout, $\log_{\bullet}(x):=\log(x)+1$.

\begin{theorem}[Efficient regularity with respect to Bohr sets]\label{thm:mainthmgen}
For all $k\geq 2$, $\mu\in (0,1)$, and $r\in (0,1]$, there are constants $F$, $D=D(k)$, and $n_0=n_0(k,\mu,r)$ such that the following holds. Suppose that $G$ is a finite abelian group of order at least $n_0$, and that $A\subseteq G$ is $k$-stable.  Then there is a regular Bohr set $B=B(K,\rho)$ which is $2\mu$-good for $A$ and which satisfies $|K|\leq (D\log_{\bullet}(r^{-1})/\mu F)^{D}$ and $r(F\mu/D\log_{\bullet}(r^{-1}))^{D}\leq \rho \leq r$.
\end{theorem}

The parameter $r$ has been introduced for the sole purpose of facilitating an easy deduction of several special cases of our structural result, see Corollaries 1-3 in Section \ref{sec:struc}. For the main application, we shall simply set $r=1$. Indeed, with $r=1$ Theorem \ref{thm:mainthmgen} immediately yields Theorem \ref{thm:main} as stated in the introduction. 

To prove Theorem \ref{thm:mainthmgen}, we closely follow the structure of the proof of \cite[Theorem 3]{Terry:5CFoQi42}, with Bohr sets in place of subspaces. It will be helpful to recall the following facts and notation from  \cite[Section 4]{Terry:5CFoQi42}.  Given an integer $n\geq 1$, let $2^n:=\{0,1\}^n$ and $2^{<n}:=\bigcup_{i=0}^{n-1} \{0,1\}^i$, where $\{0,1\}^0:=\{\langle\rangle\}$ is the set containing only the empty string, $\langle \rangle$. Given $\eta, \eta'\in 2^{<n}$, we write $\eta \trianglelefteq \eta'$ if and only if $\eta=\langle\rangle$ or $\eta$ is an initial segment of $\eta'$. Furthermore, $\eta \triangleleft \eta'$ means that $\eta \trianglelefteq \eta'$ and $\eta \neq \eta'$.  Given $\eta \in 2^n$ and $i\in \{0,1\}$, $\eta\wedge i$ denotes the element of $2^{n+1}$ obtained by adding $i$ to the end of $\eta$.  If $\eta=(\eta_1,\ldots, \eta_n) \in 2^n$ and $1\leq i\leq n$, let $\eta |_i:=(\eta_1,\ldots, \eta_i)$, and $\eta(i):=\eta_i$.  By convention, $\eta|_0:=\langle \rangle$ and $2^{<0}:=\emptyset$.  

The proof of Theorem \ref{thm:mainthmgen} proceeds by contradiction: we shall assume that there are no good Bohr sets with the desired parameters, and construct a witness of instability. Instead of constructing an instance of the order property, however, we aim to establish the existence of a \emph{tree}, in the following model-theoretic sense. 

\begin{definition}[Tree bound]\label{def:treebound}
Given a group $G$ and $A\subseteq G$, the \emph{tree bound} for $A$, denoted by $d(A)$, is the least integer $d$ such that there do not exist sequences $\langle a_{\eta}: \eta \in 2^d\rangle$, $\langle b_{\rho}: \rho \in 2^{<d}\rangle$ of elements of $G$ with the property that for each $\eta\in 2^d$ and $\rho\in 2^{<d}$, if $\rho \triangleleft \eta$, then $a_{\eta}+b_{\rho}\in A$ if and only if $\rho\wedge 1\trianglelefteq \eta$.
\end{definition}

Theorem \ref{thm:treefact} below, which is a special case of \cite[Lemma 6.7.9]{Wilfrid:1993wx}, asserts that the existence of a large tree is (at least qualitatively) equivalent to instability.

\begin{theorem}[Stability bounds height of trees]\label{thm:treefact}
For each integer $k$ there exists $d=d(k)<2^{k+2}-2$ such that if $A\subseteq G$ is a $k$-stable, then its tree bound $d(A)$ is at most $d$.  
\end{theorem}

Before we begin the proof of Theorem \ref{thm:mainthmgen}, we set up some more notation. Given $A\subseteq G$, we let $A^1:=A$, $A^0:=G\setminus A$, and for $x\in G$ and $i\in \{0,1\}$, set $N^i(x):=A^i-x$.  On several occasions we shall require the following facts, proved in \cite[Lemmas 1 and 2]{Terry:5CFoQi42}. If $A$ is $k$-stable, $B$ is $\ell$-stable, $i\in \{0,1\}$ and $x\in G$, then $A^i+x$ is $(k+1)$-stable, and $A\cap B$ is $h(k,\ell)$-stable, where $h(k,\ell):=(k+\ell)2^{k+\ell}+1$.   Given $k\geq 2$, let $f_k(y)$ be the function in one variable defined by $f_k(y):=h(k+1,y)$. For $t\geq 1$, let $f_k^t(y)$ denote the function obtained by applying $f_k$ $t$ times. That is, for any integer $t\geq -1$, $f_k(f_k^t(y))=f_k^{t+1}(y)$, where by convention we let $f_k^{-1}$ and $f_k^0$ denote the constant functions $f_k^{-1}(y):=2$ and $f^0_k(y):=k$, respectively. Observe that for any $t\geq -1$ and $i\in \{0,1\}$, if $A$ is $k$-stable and $B$ is $f_k^t(k)$-stable, then $A^i\cap B$ is $f_k^{i+1}(k)$-stable.

\vspace{3mm}

\begin{proofof}{Theorem \ref{thm:mainthmgen}} 
Fix an integer $k\geq 2$. Let $f(y):=f_k(y)$ be as above, $d=d(k)$ as in Theorem \ref{thm:treefact}, $C_1, C_2$ as in Proposition \ref{prop:theoremC}, and choose $D>\max\{C_1,C_2^{-1}, f^{d+2}(k), 6(d+2)6^{d+2}+1\}$.  Fix $\mu\in (0,1)$ and $r\in (0,1]$. Set $\epsilon:=\min\{\mu, C_2, 1/4\}$ and $F:=\min \{C_2, 1/4\}$. Choose $n_0=n_0(k,\epsilon,r )$ sufficiently large so that 
\[n_0\geq 2 \left(r(\epsilon/D\log_{\bullet}(r^{-1}))^{D}\right)^{-(D \log_{\bullet}(r^{-1})/\e)^{D}}.\]
Suppose that $G$ is a finite abelian group of size at least $n_0$ and that $A\subseteq G$ is $k$-stable. Note that $F\leq 1$ and $\e=\min\{\mu,F\}$, so that $\mu F\leq \e\leq \mu$. It therefore suffices to find a regular Bohr set $B=B(K,\rho)$ which is $2\epsilon$-good for $A$ and which satisfies $|K|\leq (D\log_{\bullet}(r^{-1})/\e)^{D}$ and $r(\e/D\log_{\bullet}(r^{-1}))^{D}\leq \rho\leq r$.

Note that if $|A|\geq (1-2\epsilon)|G|$ or $|A|\leq 2\epsilon|G|$ then taking $B=B(\emptyset, r)=G$, we are done. Thus we shall henceforth assume that $2\epsilon<|A|<(1-2\epsilon)|G|$. 

Suppose towards a contradiction that there is no regular Bohr set $B(K,\rho)$ which is $2\epsilon$-good for $A$ and which satisfies $|K|\leq (D\log_{\bullet}(r^{-1})/\e )^{D}$  and $ r(\e /\log_{\bullet}(r^{-1})D)^{D}\leq \rho \leq r$.  We now construct sequences of sets and elements indexed by elements of $2^{\leq d}$ as follows.
\begin{enumerate}[(a)]
\item $\langle B_{\eta}=B(K_{\eta},\rho_{\eta}): \eta\in 2^{ \leq d}\rangle$, where each $B_{\eta}$ is a regular Bohr set;
\item $\langle g_{\eta}: \eta\in 2^{\leq d}\rangle$ and $\langle x_{\eta}: \eta\in 2^{\leq d} \rangle$, where each $g_{\eta}$ and $x_{\eta}$ are elements of $G$; 
\item $\langle X_{\eta}: \eta\in 2^{\leq d}\rangle$, where each $X_{\eta}\subseteq G$.
\end{enumerate}
These sequences will satisfy the following for each $0\leq t\leq d$, and all $\eta \in 2^t$, where $m:=\lceil 1/\epsilon \rceil$ and $\ell(t)=(Dm\log_{\bullet}(r^{-1}))^{6^t}$.  
\begin{enumerate}[(i)]
\item $B_{\eta}=B(K_{\eta},\rho_{\eta})$ satisfies $|K_{\eta}|\leq \ell(t)$ and $ r\ell(t)^{-6t}\leq \rho_{\eta}\leq r$;\label{parameterbounds*}
\item \label{ghypp*} for each $i\in \{0,1\}$, $|N^i(g_{\eta})\cap B_{\eta}|> 2\epsilon|B_{\eta}|$;
\item \label{largeintersection*} $|X_{\eta}\cap B_{\eta}|\geq (1-\epsilon )|B_{\eta}|\geq 1$;
\item $X_{\eta}$ is $f^{t-1}(k)$-stable in $G$;\label{stabilityhypothesis}
\item if $\eta=\sigma \wedge i$, then $X_{\eta}=N^i(g_{\sigma}+x_{\sigma\wedge i})\cap (X_{\sigma}-x_{\sigma \wedge i})$;\label{defofset}
\item for all $s<t$, and $\sigma \in 2^{s}$ satisfying $\sigma \triangleleft \eta$, the following holds for all $x\in X_{\eta}$:\label{importanthypothesis}
\begin{align}\label{5}
x+g_{\sigma}+ x_{\eta|_{s+1}}+\ldots+x_{\eta|_{t}} \in A \Leftrightarrow \eta(s+1)=1.
\end{align}
\end{enumerate}

We proceed by induction on $t$. For the base case $t=0$, set $x_{\langle \rangle}:=0=:g_{\langle \rangle}$, $\rho_{\langle \rangle}:=r$, $K_{\langle \rangle}:=\emptyset$, and $X_{\langle \rangle}=B_{\langle \rangle}=G$.  Note that $X_{\langle \rangle}$ is $2=f^{-1}(k)$-stable, and since $2\epsilon|G|<|A|<(1-2\epsilon)|G|$, $|N^i(g_{\langle \rangle})\cap X_{\langle \rangle}|> 2\epsilon|B_{\langle \rangle}|$ holds for each $i\in \{0,1\}$. It is now straightforward to see that (\ref{parameterbounds*})-(\ref{importanthypothesis}) hold for all $\eta\in 2^0=\{\langle \rangle\}$ (observe that (\ref{defofset}) and (\ref{importanthypothesis}) are vacuous).  

Suppose now that $0\leq t<d$ and assume that we have inductively constructed $\langle B_{\eta}=B(K_{\eta},\rho_{\eta}): \eta\in 2^{\leq t}\rangle$, $\langle g_{\eta}: \eta\in 2^{\leq t}\rangle$, $\langle x_{\eta}: \eta\in 2^{\leq t}\rangle$, and $\langle X_{\eta}: \eta\in 2^{\leq t}\rangle$ such that (\ref{parameterbounds*})-(\ref{importanthypothesis}) hold for all $\eta \in 2^t$.  We now show how to extend each sequence by defining $B_{\eta\wedge i}, g_{\eta \wedge i}, x_{\eta\wedge i}, X_{\eta\wedge i}$ for each $\eta \in 2^t$ and $i\in \{0,1\}$.  Fix $\eta \in 2^{t}$.  Observe that our induction hypotheses (\ref{ghypp*}) and (\ref{largeintersection*}) imply that for each $i\in \{0,1\}$,
\begin{align}\label{large1**}
|N^i(g_{\eta})\cap X_{\eta}\cap B_{\eta}|>(2\epsilon-\epsilon)|B_{\eta}|\geq \epsilon|B_{\eta}|.
\end{align}
For each $i\in \{0,1\}$, because $A$ is $k$-stable, \cite[Lemma 2]{Terry:5CFoQi42} implies that $N^i(g_{\eta})=A^i-g_{\eta}$ is $(k+1)$-stable. By our induction hypothesis (\ref{stabilityhypothesis}), $X_{\eta}$ is $f^{t-1}(k)$-stable.  Consequently, for each $i\in \{0,1\}$, $N^i(g_{\eta})\cap X_{\eta}$ is $h(k+1,f^{t-1}(k))=f^{t}(k)$-stable.

Combining the fact that $N^i(g_{\eta})\cap X_{\eta}$ is $f^{t}(k)$-stable with (\ref{large1**}), allows us to conclude, by Proposition \ref{prop:theoremC}, that for each $i\in \{0,1\}$, there is a regular Bohr set $B_{\eta \wedge i}=B(K_{\eta\wedge i},\rho_{\eta\wedge i})$ and $x_{\eta \wedge i}\in G$ satisfying $K_{\eta \wedge i}\supseteq K_{\eta}$,  $\rho_{\eta \wedge i}\leq \rho_{\eta}$,
\begin{align}
&\label{align:2}|K_{\eta \wedge i}|\leq C_1(|K_{\eta}|f^t(k)\e^{-3}\log(\rho_{\eta}^{-1})+\log(\e^{-1})),\\
&\rho_{\eta\wedge i}\geq C_2\rho_{\eta} \e(C_1(|K_{\eta}|f^t(k)\epsilon^{-3}\log(\rho_{\eta}^{-1})+\log (\e^{-1})))^{-3}, \text{ and }\label{6**}\\
&|(N^i(g_{\eta})\cap X_{\eta}) \cap (B_{\eta \wedge i}+x_{\eta \wedge i})|\geq (1-\e)|B_{\eta \wedge i}|.\label{large7**}
\end{align}
For each $i\in \{0,1\}$, set $X_{\eta \wedge i}:=(N^i(g_{\eta})\cap X_{\eta})-x_{\eta\wedge i}=N^i(g_{\eta}+x_{\eta \wedge i})\cap( X_{\eta}-x_{\eta \wedge i} )$.  Observe that for each $i\in \{0,1\}$, by (\ref{align:2}) and our inductive hypothesis (\ref{parameterbounds*}) on $B_{\eta}$, we have
\begin{align*}
|K_{\eta \wedge i}|\leq C_1(|K_{\eta}|f^t(k)m^3\log(1/\rho_{\eta})+\log m) &\leq C_1(\ell(t)f^t(k)m^3\log(r^{-1}\ell(t)^{6t})+\log m)\\
&\leq 2C_1(\ell(t)f^t(k)m^3\log(r^{-1}\ell(t)^{6t}))\\
&\leq12tC_1\ell(t)f^t(k)m^3\log(r^{-1}\ell(t)).
\end{align*}
Using the fact that $D>\max\{f^t(k), 12t,C_1\}$ and $\log(r^{-1}\ell(t))\leq \log(r^{-1})\ell(t)\leq  \log_{\bullet}(r^{-1})\ell(t)$, and plugging in the definition of $\ell(t)$, we obtain that
\begin{align*}
12tC_1\ell(t)f^t(k)m^3\log(r^{-1}\ell(t))\leq D^3\ell(t)^2m^3\log_{\bullet}(r^{-1})&\leq (Dm\log_{\bullet}(r^{-1}))^{2\cdot 6^t+3}\leq \ell(t+1),
\end{align*}
where the last inequality holds because $2\cdot 6^t+3\leq 6^{t+1}$. Thus we have shown that
\begin{align}\label{align:3}
|K_{\eta\wedge i}|\leq C_1|K_{\eta}|(f^t(k)m^3\log(1/\rho_{\eta})+\log m)\leq \ell(t+1).
\end{align}
Combining (\ref{6**}) with the fact that $C_2\geq m^{-1}$, as well as (\ref{align:3}) and our induction hypothesis (\ref{parameterbounds*}) yields the lower bound
\begin{align*}\label{align:4}
\rho_{\eta\wedge i}\geq C_2\rho_{\eta}m^{-1}(C_1(|K_{\eta}|f^t(k)m^3\log(\rho_{\eta}^{-1})+\log m))^{-3}&\geq m^{-2}\rho_{\eta}\ell(t+1)^{-3}\\
&\geq m^{-2}r\ell(t)^{-6t}\ell(t+1)^{-3}\\
&\geq r(Dm\log_{\bullet}(r^{-1}))^{-(6t\cdot 6^t+3\cdot 6^{t+1}+2)}\\
&\geq  r\ell(t+1)^{-6(t+1)},
\end{align*}
where the last inequality holds because $6t\cdot 6^t+3\cdot 6^{t+1}+2\leq 6(t+1)6^{t+1}$.  By definition, $6^{t+1}<D$, so $|K_{\eta \wedge i}|\leq \ell(t+1)<(mD\log_{\bullet}(r^{-1}))^{D}$. Since also $6(t+1)6^{t+1}<D$, we have $\rho_{\eta\wedge i}\geq r\ell(t+1)^{-6(t+1)}\geq r(mD\log_{\bullet}(r^{-1}))^{-D}$. Further, by our induction hypothesis (i) and our choice of $\rho_{\eta \wedge i}$ we have $\rho_{\eta \wedge i}\leq \rho_{\eta}\leq r$.  Therefore, $B_{\eta\wedge i}$ cannot be $2\epsilon$-good.  Consequently,  there is $g_{\eta\wedge i}\in G$ such that for each $j\in \{0,1\}$, $|N^j(g_{\eta\wedge i})\cap B_{\eta\wedge i}|> 2\epsilon|B_{\eta\wedge i}|$.  We have thus shown how to define $B_{\tau}, x_{\tau}, g_{\tau}, X_{\tau}$ for each $\tau \in 2^{t+1}$. 

It remains to prove that (\ref{parameterbounds*})-(\ref{importanthypothesis}) hold for each $\tau\in 2^{t+1}$.  Fix $\tau\in 2^{t+1}$, so $\tau=\eta\wedge i$ for some $i\in \{0,1\}$ and $\eta \in 2^t$.  Note that (\ref{parameterbounds*}) holds by the above,  (\ref{ghypp*}) holds by our choice of $g_{\eta \wedge i}$, and (\ref{defofset}) holds by definition of $X_{\eta \wedge i}$. For (\ref{largeintersection*}), observe that (\ref{large7**}) together with the definition of $X_{\eta\wedge i}$ implies that $|X_{\tau}\cap B_{\tau}|\geq (1-\epsilon)|B_{\tau}|$.  That $(1-\epsilon)|B_{\tau}|\geq 1$ follows from Lemma \ref{lem:sizefact} and the facts that $\rho_{\tau}\geq r (\epsilon/D\log_{\bullet}(r^{-1}))^D$ and $|K_{\tau}|\leq (D\log_{\bullet}(r^{-1})/\e)^{D}$, together with the assumption $|G|\geq n_0$.   For (\ref{stabilityhypothesis}), note that $X_{\tau}=(N^i(g_{\eta})\cap X_{\eta})-x_{\eta\wedge i}$ is just a translate of $N^i(g_{\eta})\cap X_{\eta}$.  We know that $N^i(g_{\eta})\cap X_{\eta}$ is $f^{t}(k)$-stable, so \cite[Lemma 1]{Terry:5CFoQi42} implies that $X_{\tau}$ is $f^{t}(k)$-stable.   Finally, the verification of (\ref{importanthypothesis}) follows the argument in the proof of \cite[Theorem 3]{Terry:5CFoQi42} verbatim.

We may thus assume that we have constructed sequences (a)-(c) satisfying properties (\ref{parameterbounds*})-(\ref{importanthypothesis}). It remains to construct the tree of height $d$ that acts as a witness to instability. For each $\eta \in 2^d$ choose some $c_{\eta}\in X_{\eta}$ (such a $c_{\eta}$ exists by (iii)) and set $a_{\eta}:=c_{\eta}+\sum_{\sigma\trianglelefteq \eta}x_{\sigma}$.  Let $b_{\langle \rangle}:=g_{\langle \rangle}=0$ and for each $0<s<d$ and $\sigma \in 2^{s}$, let $b_{\sigma}:= g_{\sigma}-x_{\sigma|_{1}}-\ldots -x_{\sigma|_{s-1}}-x_{\sigma}$.  Then it is easily check (see the proof of \cite[Theorem 3]{Terry:5CFoQi42}) that for all $0\leq s<d$ and all $\sigma\in 2^s$, $\eta\in 2^d$, if $\sigma \triangleleft \eta$, then $a_{\eta}+b_{\sigma}=c_\eta+g_\sigma+x_{\eta|s+1}+\dots+x_{\eta|d}\in A$ if and only if $\sigma \wedge 1\trianglelefteq \eta$.  Definition \ref{def:treebound} now implies that $d(A)> d$, contradicting the initial assumption that $A$ is $k$-stable. 
 \end{proofof}

\section{Structural interlude}\label{sec:struc}

In order to obtain a description of stable subsets of a general finite abelian group, we prove an important lemma that elucidates the structure of the set of translates of the good Bohr set found in Theorem \ref{thm:mainthmgen} that are almost entirely filled by $A$. This lemma will be used again in Section \ref{sec:boot}.

\begin{lemma}[Structure of almost-full translates]\label{lem:strucI}
Let $k\geq 2$ and $\epsilon\in (0,1/3)$. Let $n_1=n_1(k,\epsilon)=n_0(k,\e/2,1)$, where $n_0$ is as in Theorem \ref{thm:mainthmgen}. Suppose that $G$ is a finite abelian group of order at least $n_1$, and that $A\subseteq G$ is $k$-stable. Let $B$ be the regular, $\e$-good Bohr set for $A$ given by Theorem \ref{thm:mainthmgen} applied with $\mu=\e/2$ and $r=1$, and let 
\[I:=\{x\in G: |(B+x)\cap A|\geq (1-\epsilon)|B|\}.\]  
Then, provided that $I$ is non-empty, we have $I=I+\langle B'\rangle$ for any Bohr set $B'$ such that $B'\prec_{\epsilon}B$.
\end{lemma}
\begin{proof}
Suppose that $G$ is a finite abelian group of order at least $n_1$, and that $A\subseteq G$ is $k$-stable.  Let $B=B(K,\rho)$ be the regular, $\epsilon$-good Bohr set given as a result of applying Theorem \ref{thm:mainthmgen} with $\mu=\e/2$ and $r=1$.  Suppose that $B'=B(K,\rho')$ is such that $B'\prec_{\epsilon}B$. Let $B^+=B(K,\rho+\rho')$. Clearly $I\subseteq I+\langle B' \rangle$, so it remains to show that $I+\langle B' \rangle\subseteq I$. 

Fix $x+b\in I+\langle B'\rangle$.  Then $b=b_1+\cdots +b_m$, for some $m\in \mathbb{N}$ and $b_1,\ldots, b_m\in B'$.  We induct on $m$. Consider first the case $m=1$. In this case we have $x+b_1\in I+B'$, so that
\begin{equation}\label{eq:regfact}
|(B+x+b_1)\cap (B+x)|=|(B+b_1)\cap B|\geq |B|-|B^+\setminus B|\geq (1-\epsilon)|B|,
\end{equation}
where the last inequality holds by (\ref{eq:regfact1}) following Definition \ref{def:bohrint}. We are now ready to show that $A$ fills $B+b_1+x$ almost entirely. Observe that 
\[|(B+x+b_1)\cap A|\geq |(B+x+b_1)\cap (B+x)\cap A|\geq |(B+x+b_1)\cap (B+x)|-|(B+x)\setminus A|,\]
which in turn is bounded below by
\[(1-\epsilon)|B|-\epsilon |B|=(1-2\epsilon)|B|,\]
using (\ref{eq:regfact}) and the fact that $x\in I$. Since $\epsilon <1/3$, this shows that 
\[|(B+x+b_1)\cap A|\geq (1-2\epsilon)|B|>\epsilon |B|.\]  
But since $B$ is $\epsilon$-good, we must have $|(B+x+b_1)\cap A|\geq (1-\epsilon)|B|$, so $x+b_1\in I$.   

Suppose now that $m>1$, and assume we have shown that for any $1\leq m'<m$ and any $b_1',\ldots, b_{m'}'\in B'$, $x+b_1'+\cdots +b_{m'}'\in I$. Let $b_2':=b_2+\cdots +b_m$.  By the induction hypothesis, both $x+b_1$ and $x+b_2'$ lie in $I$.
Note that, since $b_1\in B'$,
\begin{equation}\label{eq:regfact2}
|(B+x+b_1+b'_2)\cap(B+x+b'_2)|=|(B+b_1)\cap B|\geq (1-\e)|B|,
\end{equation}
and thus by the same reasoning as above,
\begin{align*}
|(B+x+b_1+b'_2)\cap A|&\geq |(B+x+b_1+b'_2)\cap (B+x+b_2')\cap A|\\
&\geq |(B+x+b_1+b'_2)\cap (B+x+b'_2)|-|(B+x+b'_2)\setminus A|,
\end{align*}
which is at least $(1-\epsilon)|B|-\epsilon |B|=(1-2\e)|B|$, by (\ref{eq:regfact2}) and since $x+b'_2\in I$. It follows as before that $|(B+x+b_1+b'_2)\cap A|>\epsilon |B|$, so since $B$ is $\e$-good, $|(B+x+b_1+b'_2)\cap A|\geq (1-\epsilon)|B|$, i.e. $x+b_1+b'_2\in I$. This completes the inductive step, and thus the proof that $I=I+\langle B' \rangle$. 
\end{proof}

Lemma \ref{lem:strucI} enables us to prove a version of Theorem \ref{thm:mainthmgen} for finite-dimensional vector spaces over a prime field in which the good Bohr set is replaced by a good subgroup. Note that the assumption on $G$ in Corollary \ref{cor:vectorspaces} below is simply that its order be large, so it encompasses the setting of $G=\F_p^n$ for a small fixed prime $p$ and very large $n$, as well as the case when $G=\Z/p\Z$ is a large cyclic group of prime order. 

\begin{corollary}[Efficient regularity with respect to subgroups in $\F_p^n$]\label{cor:vectorspaces}
For all $k\geq 2$ and $\e\in (0,1)$ there is $n_2=n_2(k,\e)$ such that the following holds. Suppose that $G$ is a finite-dimensional vector space over a prime field with $|G|\geq n_2$, and that $A\subseteq G$ is $k$-stable.  Then there is a subgroup $H\leq G$ which is $\e$-good for $A$ and of index at most $\exp(O_k(\e^{-O_k(1)}))$. 
\end{corollary}
\begin{proof}
 Let $F$ and $D=D(k)$ be as in Theorem \ref{thm:mainthmgen}.  Fix $\e\in (0,1)$ and set $\mu:=\e/3$.  Set $n_2:=\max\{n_0(k,\mu/2,1),n_0(k,\mu/2,\exp(-(400 D/\mu F)^{4D}-1))\}$, where $n_0$ is as in Theorem \ref{thm:mainthmgen}.\footnote{Note that for any $0<\e'\leq \e$, $0<r'\leq r\leq 1$, $n_0(k,\e',r')\geq n_0(k,\e,r)$.} Let $G:=\F_p^n$ for some prime $p$ and an integer $n$, and suppose that $|G|\geq n_2$ and that $A\subseteq G$ is $k$-stable.  It suffices to show there is a $\mu$-good subgroup with index at most $\exp(O_k(\mu^{-O_k(1)}))$.

Let us first assume that $(400D/\mu F)^{4D}< \log p$.  Applying Theorem \ref{thm:mainthmgen} with $k$, $r=1$ and $\mu/2$ in place of $\mu$, we obtain a $\mu$-good regular Bohr set $B=B(K,\rho)$ with $|K|\leq (2D/\mu F)^D$ and $\rho \geq (F\mu/2D)^D$.  By Lemma \ref{lem:regexist} we may choose $B'=B(K,\rho')\prec_{\mu}B$ such that $\rho'\geq \mu \rho/400|K|$.  By Lemma \ref{lem:sizefact}, $|B'|\geq \rho'^{|K|}|G|$, which implies that
\[\log \Big(\frac{|G|}{|\langle B'\rangle|}\Big)\leq |K|\log\Big(\frac{400|K|}{\mu \rho}\Big)\leq \frac{400|K|^2}{\mu \rho}\leq \frac{400(2D/\mu F)^{2D}}{\mu (F\mu/2D)^D}\leq(400D/\mu F)^{4D}.\]
It follows from the assumed lower bound on $\log p$ that the index of $\langle B' \rangle$ in $G$ is strictly less than $p$, and consequently, $\langle B'\rangle =G$. 

Now let $I:=\{x\in G: |A\cap (B+x)|\geq (1-\mu)|B|\}$.  By Lemma \ref{lem:strucI}, either $I=\emptyset$ or $I=I+\langle B'\rangle =G$.  If $I$ is empty, that is, 
\[|(B+x)\cap A|\leq \mu|B|\]
for all $x\in G$, then we have
\[|A|=\sum_{y\in G}1_A(y)=\frac{1}{|B|}\sum_{x\in G}\sum_{y\in G}1_A(y)1_{B+x}(y)=\frac{1}{|B|}\sum_{x\in G}|A\cap (B+x)|\leq \mu |G|.\]
If, on the other hand, $I=G$ then 
\[|(B+x)\cap A|\geq (1-\mu)|B|\]
for all $x\in G$. It follows that 
\[|A|=\frac{1}{|B|}\sum_{x\in G}|A\cap (B+x)|\geq (1-\mu)|G|,\]
and thus $|G\setminus A|\leq \mu|G|$.  Thus we have that either for all $x\in G$, $|(A-x)\cap G|=|A|\leq \mu|G|$ or for all $x\in G$, $|(A-x)\cap G|=|A|\geq (1-\mu)|G|$. In other words, $G$ is a subgroup of index $1$ which is $\mu$-good for $A$.

It remains to examine the case when $(400D/\mu F)^{4D}\geq \log p$.   Note that in this case, $(p+1)^{-1}\geq (2p)^{-1}\geq \exp(-(400 D /\mu F)^{4D}-1)$, so $|G|\geq n_0(k,\mu/2,(p+1)^{-1})$. Applying Theorem \ref{thm:mainthmgen} with $k$, $r=(p+1)^{-1}$ and $\mu/2$ in place of $\mu$, we obtain a regular Bohr set $B=B(K,\rho)$ which is $\mu$-good for $A$ satisfying 
\begin{align}\label{align:5}
|K|\leq (2D\log_{\bullet}(r^{-1}) /\mu F)^D \text{ and }\rho\geq r( F\mu/2D\log_{\bullet}(r^{-1}))^D.
\end{align}
Since $\log_\bullet(r^{-1})=\log_\bullet(p+1)\leq 2\log_\bullet p$ and $r=(p+1)^{-1}\geq (2p)^{-1}$, we have
\begin{align*}
&|K|\leq (4D\log_\bullet p/\mu F)^D\text{ and }\rho\geq (2p)^{-1}(\mu F/4 D\log_\bullet p)^D. 
\end{align*}
Combining these bounds with our assumption on $\log p$, we find that $|K|\leq O_k(\mu^{-O_k(1)})$ and $\rho\geq \exp(-O_k(\mu^{-O_k(1)})).$ By Lemma \ref{lem:sizefact} therefore, 
\begin{align*}
|G|/|B|\leq \rho^{-|K|}=\exp(-O_k(\mu^{-O_k(1)}) )^{-O_k(\mu^{-O_k(1)} )}=\exp(O_k(\mu^{-O_k(1)})).
\end{align*}
Moreover, observe that since $\rho<1/p$, $B$ is a subgroup of $G$, which completes the proof.  
\end{proof}

Corollary \ref{cor:vectorspaces} immediately implies \cite[Theorem 1]{Terry:5CFoQi42}, the main result of that paper. By the same argument as in \cite{Terry:5CFoQi42} we therefore recover the following description of the structure of stable subsets of $\F_p^n$ for a fixed prime $p$ and sufficiently large dimension $n$ \cite[Corollary 1]{Terry:5CFoQi42}.

\begin{corollary}[Stable sets in $\F_p^n$ look like a union of cosets]\label{cor:vecspace}
For every $k\geq 2$, $\e\in (0,1)$, and every prime $p$, there is $m_0=m_0(k,\e,p)$ such that the following holds. Suppose that $G=\F_p^n$ is a vector space of dimension $n\geq m_0$, and that $A\subseteq G$ is $k$-stable. Then there is a subspace $H\leqslant G$ of codimension at most $O_k(\e^{-O_k(1)})$ and a set $J\subseteq G/H$ such that $|A\Delta \cup_{H+g\in J}(H+g)|\leq \e|G|$.
\end{corollary}

By the model-theoretic arguments in \cite{Conant:2017uw} we know that a bound of the form $\e|H|$ on the symmetric difference in Corollary \ref{cor:vecspace} can be obtained, but our proof does not readily yield this stronger result.

Corollary \ref{cor:vectorspaces} also yields a structural result for cyclic groups of large prime order. Indeed, it asserts that any stable subset in such a group must be either negligibly small or almost the whole group.

\begin{corollary}[Stable sets in prime cyclic groups are essentially trivial]\label{cor:zp}
For all $\e\in (0,1)$ and $k\geq 2$, there is $p_0=p_0(k,\e)$ such that the following holds. Suppose that $G=\Z/p\Z$ is a cyclic group of prime order $p\geq p_0$, and that $A\subseteq G$ is $k$-stable. Then either $|A|\leq \e|G|$, or $|G\setminus A|\leq \e|G|$.
\end{corollary}

\begin{proof}
Fix $\e\in (0,1)$, and let $n_1(k,\e)$ be as in Corollary \ref{cor:vectorspaces}. Choose $p_0\geq n_1$ sufficiently large so that $p_0$ is larger than the term $\exp(O_k(\e^{-O_k(1)}))$ in the conclusion of Corollary \ref{cor:vectorspaces}.  Suppose that $G=\Z/p\Z$ is a cyclic group of prime order $p\geq p_0$, and that $A\subseteq G$ is $k$-stable.  Applying Corollary \ref{cor:vectorspaces}, there is a subgroup $H$ of $G$ which is $\e$-good for $A$ and which has index at most $\exp(O_k(\e^{-O_k(1)}))$. By our choice of $p_0$, we must have $H=G$. Since $G$ is $\e$-good for $A$, either $|A\cap G|=|A|\leq \e|G|$ or $|G\setminus A|\leq \e |G|$.
\end{proof}

The proof of Corollary \ref{cor:vectorspaces} relied upon the fact that in a finite-dimensional vector space over a prime field the exponent of the group equals the smallest index of a proper subgroup. In a general finite abelian group this is not always the case. In the next section we therefore take a step back and revisit Proposition \ref{prop:theoremC} in view of the insight gained in Lemma \ref{lem:strucI}.

\section{Bootstrapping the argument}\label{sec:boot}

Lemma \ref{lem:strucI} from the preceding section allows us to bootstrap our argument using the following proposition.  

\begin{proposition}[Stable sets are dense on subgroups]\label{prop:densesubgroups}
For all $\epsilon\in (0,1/3)$ and $k\geq 2$, there are constants $C=C(k)>0$ and $n_3=n_3(k,\epsilon)$ such that the following holds.  Suppose that $A\subseteq G$ is $k$-stable and that $H\leq G$ satisfies $|H|\geq n_3$ and $|A\cap H|> \epsilon |H|$. Then there exists $x\in G$ and $H'\leq H$ of index at most $\exp(\epsilon^{-C})$ in $H$ such that $|A\cap (H'+x)|\geq (1-\epsilon)|H'|$.
\end{proposition}

We will need the following fact for the proof of Proposition \ref{prop:densesubgroups}. 

\begin{lemma}\label{lem:stableonsubgroups}
If $A\subseteq G$ is $k$-stable and $H\leq G$, then $A\cap H$ is $k$-stable as a subset of $H$.
\end{lemma}
\begin{proof}
By \cite[Lemma 3]{Terry:5CFoQi42}, $A\cap H$ is $k$-stable in $G$. By definition this implies that $A\cap H$ is $k$-stable in $H$.
\end{proof}

\begin{proofof}{Proposition \ref{prop:densesubgroups}}
Fix $k\geq 2$, $\epsilon \in (0,1)$, and set $\mu:=\epsilon/2$. Let $F$ and $D=D(k)$ be as in Theorem \ref{thm:mainthmgen}. Choose $n_3$ equal to $n_0(k,\mu,1)$ given by Theorem \ref{thm:mainthmgen}.  Suppose that $A\subseteq G$ is $k$-stable, $H\leq G$, $|H|\geq n_3$, and that $|A\cap H|> \epsilon |H|$.  By Lemma \ref{lem:stableonsubgroups}, $A\cap H$ is $k$-stable in $H$, so applying Theorem \ref{thm:mainthmgen} to $A\cap H$, $H$, $\mu$ and $r=1$ yields a regular Bohr set $B=B(K,\rho)$ in $H$ which is $\epsilon$-good for $A\cap H$ as a subset of $H$ and which satisfies $|K|\leq (D/\mu F)^D$ and $\rho\geq (\mu F/D)^D$.  

Suppose that $B'$ is a Bohr set such that $B'\prec_{\epsilon}B$. Let $I:=\{y\in H: |(A\cap H)\cap (B+y)|\geq (1-\epsilon)|B|\}$ as in the statement of Lemma \ref{lem:strucI}. We already saw in the proof of Corollary \ref{cor:vectorspaces} that if $I=\emptyset$, then $|A\cap H|\leq \e|H|$ (this part of the proof did not use any of the other assumptions). We therefore conclude that $I$ is non-empty. Setting $H':=\langle B'\rangle$, Lemma \ref{lem:strucI} now tells us that $I=I+H'$. 

It remains to show that $A\cap H$ fills some coset of $H'$ almost completely. Fix $z\in I$ and consider the bipartite graph with vertex set $V:=U\sqcup W$, where $U:=B$ and $W:=H'+z$, and with edge set $E:=\{uw: u+w\in A\}$. Since $B$ is $\epsilon$-good for $A\cap H$ in $H$ and $H'+z\subseteq I$, we have that for all $w\in W$, $|N(w)\cap U|\geq (1-\epsilon)|B|$.  Consequently, $|E|\geq (1-\epsilon)|B||H'|$.  Thus there is $u\in U$ such that $|N(u)\cap W|\geq (1-\epsilon)|H'|$. In other words, $|(A\cap H)\cap (H'+z+u)|\geq (1-\epsilon)|H'|$, so we set $x:=z+u$ to obtain the desired translate of $H'$.

We only need to check that the index of $H'$ satisfies the stated bounds. If $B'=B'(K',\rho')$ then we know $|H'|\geq |B'|\geq (\rho')^{|K'|}|H|$, so the index of $H'$ in $H$ is at most $(\rho')^{-|K'|}$.  By Lemma \ref{lem:regexist}, we may assume that $\rho'\geq \epsilon \rho/400|K|\geq \e (\e F/2D)^{2D}/400$. Since $K=K'$,  $|K'|\leq (2D/\e F)^D$.  It follows that the index of $H$ is at most 
$$
(\rho')^{-|K'|}\leq (\e (\e F/2D)^{2D}/400)^{-(2D/\e F)^D} \leq \exp(\epsilon^{-C})
$$
for some constant $C=C(k)$.
\end{proofof}

Equipped with Proposition \ref{prop:densesubgroups}, we can now prove the following refinement of Theorem \ref{thm:mainthmgen}.

\begin{theorem}[Efficient regularity with respect to subgroups]\label{thm:maingen2}
For all $k\geq 2$ and $\mu\in (0,1)$, there are constants $M=M(k)$ and $n_4=n_4(k,\mu)$ such that the following holds. Suppose that $G$ is a finite abelian group of order at least $n_4$, and that $A\subseteq G$ is $k$-stable. Then there exists a $\mu$-good subgroup $H\leq G$ of index at most $\exp(\mu^{-M})$.
\end{theorem}

Theorem \ref{thm:maingen2} is equivalent to Theorem \ref{thm:mainintro2} as stated in the introduction. Before re-running the tree argument to prove Theorem \ref{thm:maingen2}, let us deduce the following structural consequence valid in any finite abelian group.

\begin{corollary}[Structure of stable sets in a finite abelian group]\label{cor:gengroup}
For all $\epsilon\in (0,1)$ and $k\geq 2$, there is $n_5=n_5(k,\epsilon)$ such that the following holds. Suppose that $G$ is a finite abelian group of order at least $n_5$, and that $A\subseteq G$ is $k$-stable. Then there is a subgroup $H\leq G$ of index at most $\exp(\epsilon^{-O_k(1)})$ and a set $J\subseteq G/H$ such that $|A\Delta \bigcup_{H+g\in J} (H+g)|\leq \epsilon |G|$.
\end{corollary}

The proof of Corollary \ref{cor:gengroup} is near-identical to that of \cite[Corollary 1]{Terry:5CFoQi42}. We include it here for the sake of completeness.

\begin{proof}
Let $n_5=n_4(k,\e)$ as in Theorem \ref{thm:maingen2}, which implies that there is a subgroup $H\leqslant G$ of index at most $\exp(\epsilon^{-O_k(1)})$ which is $\e$-good for $A$.  This means that for all $g\in G$, either $|(A-g)\cap H|=|A\cap (H+g)|\leq \epsilon |H|$ or $|(A-g)\cap H|=|A\cap (H+g)|\geq (1-\epsilon) |H|$.  Let $J:=\{H+g\in G/H: |A\cap (H+g)|\geq (1-\epsilon) |H|\}$, and let 
\begin{align*}
X:=\bigcup_{H+g\in J}(H+g)\text{ and }Y:=\bigcup_{H+g\in (G/H)\setminus J}(H+g). 
\end{align*}
Then by definition of $X$, $Y$, $J$ and $\epsilon$-goodness of $H$, we have that 
\[|A\setminus X|=|A\cap Y|\leq (\epsilon |H|)|(G/H)\setminus J|=\epsilon|H|(|G|/|H|-|J|),\]
as well as
\[|X\setminus A|\leq \epsilon |H||J|. \]
Thus $|A\Delta X|\leq \epsilon|H|(|G|/|H|-|J|)+\epsilon |H||J| = \epsilon |H|(|G|/|H|)=\epsilon |G|$ as desired.
\end{proof}

\begin{proofof}{Theorem \ref{thm:maingen2}} For $\ell\geq 2$, let $C(\ell)$ be the constant depending on $\ell$ given by Proposition \ref{prop:densesubgroups}.  Let $d=d(k)$ be as in Theorem \ref{thm:treefact}.  For each $0\leq t\leq d+2$, let $M_t:=\sum_{i=-1}^tC(f^i(k))$ and set $M:=M_{d+2}$.  Fix $\mu\in (0,1)$ and let $\epsilon:=\mu/8$.  Let $n_4$ be sufficiently large so that $n_4>\exp(\epsilon^{-M})\max\{n_3(\epsilon, f^{t-1}(k))+2: 0\leq t\leq d\}$, where $n_3$ is as in Proposition \ref{prop:densesubgroups}. Suppose that $G$ is an abelian group of order at least $n_4$, and that $A\subseteq G$ is $k$-stable.  We shall find a subgroup $H\leq G$ which is $2\epsilon$-good for $A$ and which has index at most $\exp(\e^{-M})$.

If $|A|\leq 2\e|G|$ or $|A|\geq (1-2\e)|G|$, then we are done by taking $H=G$. So we may assume throughout the proof that $2\e < |A|< (1-2\e)|G|$.

Suppose towards a contradiction that there is no subgroup of index at most $\exp(\e^{-M})$ which is $2\epsilon$-good for $A$.  We now simultaneously construct four sequences $B_\eta$, $g_\eta$,$x_\eta$, $X_\eta$, indexed by elements of $2^{\leq d}$, that satisfy conditions (a'), (b'), (c'), where (b') and (c') are identical to (b) and (c) in the proof of Theorem \ref{thm:mainthmgen}, and (a') is given by 
\begin{enumerate}[(a')]
\item $\langle B_{\eta}: \eta\in 2^{ \leq d}\rangle$, where each $B_{\eta}\leq G$.
\end{enumerate}
For each $0\leq t<d$ and $\eta\in 2^t$, the sequences (a')-(c') will satisfy conditions labeled (i')-(vi'), where (ii')-(vi') are identical to (ii)-(vi) in the proof of Theorem \ref{thm:mainthmgen} and (i') is given by
\begin{enumerate}[(i')]
\item $B_{\eta}\leq G$ has index at most $\exp(\e^{-M_{t-1}})$.\label{parameterbounds}
\end{enumerate}
We proceed by induction on $t$. For the base case $t=0$, set $x_{\langle \rangle}:=0$, $g_{\langle \rangle}:=0$, $B_{\langle \rangle}:=G$, and $X_{\langle \rangle}:=G$.  As in the proof of Theorem \ref{thm:mainthmgen}, it is straightforward to verify that (i')-(vi') hold for all $\eta\in 2^0=\{\langle \rangle\}$.  

Suppose now that $0\leq t<d$ and assume we have inductively constructed $\langle B_{\eta}: \eta\in 2^{\leq t}\rangle$, $\langle g_{\eta}: \eta\in 2^{\leq t}\rangle$, $\langle x_{\eta}: \eta\in 2^{\leq t}\rangle$, and $\langle X_{\eta}: \eta\in 2^{\leq t}\rangle$ such that (i')-(vi') hold for all $\eta \in 2^t$.  We now show how to extend each sequence by defining $B_{\eta\wedge i}, g_{\eta \wedge i}, x_{\eta\wedge i}, X_{\eta\wedge i}$ for each $\eta \in 2^t$ and $i\in \{0,1\}$.  Fix $\eta \in 2^{t}$.  Observe that our induction hypotheses (\ref{ghypp*}') and (\ref{largeintersection*}') imply that for each $i\in \{0,1\}$,
\begin{align}\label{large1*}
|N^i(g_{\eta})\cap X_{\eta}\cap B_{\eta}|>(2\epsilon-\epsilon)|B_{\eta}|\geq \e|B_{\eta}|.
\end{align}
The same argument as in the proof of Theorem \ref{thm:mainthmgen} implies that, for each $i\in \{0,1\}$, the set $N^i(g_{\eta})\cap X_{\eta}$ is $h(k+1,f^{t-1}(k))=f^{t}(k)$-stable.

The inductive hypothesis (i') on $B_{\eta}$ and our choice of $n_4$ imply that $|B_{\eta}|\geq n_4/\exp(\e^{-M})>n_3(f^{t}(k),\epsilon)$. Combining this with the fact that $N^i(g_{\eta})\cap X_{\eta}$ is $f^{t}(k)$-stable and (\ref{large1*}), we see that Proposition \ref{prop:densesubgroups} implies that for each $i\in \{0,1\}$, there is a subgroup $B_{\eta \wedge i}$ of $B_{\eta}$ and $x_{\eta \wedge i}\in G$ such that $B_{\eta \wedge i}$ has index at most $\exp(\epsilon^{-C(f^t(k))})$ in $B_{\eta}$ and
\begin{align}
|(N^i(g_{\eta})&\cap X_{\eta}) \cap (B_{\eta \wedge i}+x_{\eta \wedge i})|\geq (1-\epsilon)|B_{\eta \wedge i}|.\label{large7*}
\end{align}
For each $i\in \{0,1\}$, set $X_{\eta \wedge i}:=(N^i(g_{\eta})\cap X_{\eta})-x_{\eta\wedge i}=N^i(g_{\eta}+x_{\eta \wedge i})\cap( X_{\eta}-x_{\eta \wedge i} )$.  The bound on the index of $B_{\eta \wedge i}$ in $B_{\eta}$ together with our inductive hypothesis (\ref{parameterbounds}') on $B_{\eta}$ yields that for each $i\in \{0,1\}$, the index of $B_{\eta \wedge i}$ in $G$ is at most $\exp(\e^{-M_{t-1}-C(f^t(k))})=\exp(\e^{-M_t})$, establishing (\ref{parameterbounds}') for $B_{\eta\wedge i}$. Since the index of $B_{\eta\wedge i}$ is less than $\exp(\e^{-M})$, $B_{\eta\wedge i}$ cannot be $2\epsilon$-good.  Consequently,  there is $g_{\eta\wedge i}\in G$ such that for each $j\in \{0,1\}$, $|N^j(g_{\eta\wedge i})\cap B_{\eta\wedge i}|>2\epsilon|B_{\eta\wedge i}|$.  This completes our construction of $B_{\tau}, x_{\tau}, g_{\tau}, X_{\tau}$ for each $\tau \in 2^{t+1}$. 

We have already shown that (i') holds for $B_{\tau}$ for all $\tau\in 2^{t+1}$, and the verifications of (ii') and (iv')-(vi') are identical to the proof of Theorem \ref{thm:mainthmgen}. For (iii'), $|X_{\tau}\cap B_{\tau}|\geq (1-\e)|B_{\tau}|$ follows from (\ref{large7*}) and the definition of $X_{\tau}$. The inequality $(1-\e)|B_{\tau}|\geq 1$ holds because the index of $B_{\tau}$ is at most $\exp(\e^{-M})$ and $|G|\geq n_4$.   Thus we may assume that we have constructed sequences (a'), (b'), (c') satisfying properties (\ref{parameterbounds}')-(\ref{importanthypothesis}'). Exactly as in the proof of \cite[Theorem 3]{Terry:5CFoQi42}, this allows us to show that the tree bound $d(A)$ satisfies $d(A)>d$, contradicting the initial assumption that $A$ is $k$-stable.
\end{proofof}

\providecommand{\bysame}{\leavevmode\hbox to3em{\hrulefill}\thinspace}
\providecommand{\MR}{\relax\ifhmode\unskip\space\fi MR }
\providecommand{\MRhref}[2]{%
  \href{http://www.ams.org/mathscinet-getitem?mr=#1}{#2}
}
\providecommand{\href}[2]{#2}


\end{document}